\documentclass[11pt]{amsart}
\usepackage{geometry} 
\usepackage{hyperref}
\usepackage{graphicx}
\usepackage{amssymb}
\usepackage{epstopdf}
\usepackage{amscd}
\title{INTRO}
\author{}

\newcommand{\ga}{\gamma}

\newcommand{\om}{\omega}
\newcommand{\Om}{\Omega}
\newcommand{\pa}{\partial}

\theoremstyle{plain}
\newtheorem{thm}{Theorem}

\newtheorem{prop}[thm]{Proposition}

\newtheorem*{qu}{Question}

\theoremstyle{definition}
\newtheorem{df}{Definition}
\newtheorem*{rem}{Remark}

\begin{document}

\title{Closed symmetric 3-differentials on complex surfaces}
\author{Federico Buonerba}
\address{Courant Institute of Mathematical Sciences, New York University, New York City, NY 10012-1185}
\email{buonerba@cims.nyu.edu}
\author{Dmitry Zakharov}
\address{Courant Institute of Mathematical Sciences, New York University, New York City, NY 10012-1185}
\email{dvzakharov@gmail.com}

\begin{abstract}

We give a necessary and sufficient condition for a non-degenerate symmetric $3$-differential with nonzero Blaschke curvature on a complex surface to be locally representable as a product of three closed holomorphic $1$-forms. We give two versions of this condition corresponding to different choices of coordinates, one of which defines a coordinate-free differential operator, answering a question of Bogomolov and de Oliveira.

\end{abstract}

\maketitle

In \cite{1}, \cite{2}, \cite{3} Bogomolov and de Oliveira initiated a systematic study of closed symmetric differentials---that is, symmetric differentials that can be locally decomposed as a product of closed 1-forms---on complex projective manifolds. This notion is a natural higher rank extension of the notion of a closed differential 1-form.

The existence of nontrivial symmetric differentials on a projective variety is related to its topology: a variety $X$ has a nontrivial $1$-form if and only $\pi_1(X)$ has infinite abelianization, and it was recently shown in \cite{5} that $X$ has a nontrivial symmetric differential of some degree if $\pi_1(X)$ has a finite-dimensional representation with infinite image. The existence of closed symmetric differentials imposes much stronger topological restrictions on the underlying manifold. Any closed symmetric differential of rank 2 on a projective variety $X$ is obtained as pullback under a morphism to a finite quotient of an abelian variety, and in particular, $\pi_1(X)$ contains a finite-index subgroup with infinite abelianization (Thm. 3.2 in \cite{1}). Furthermore, a global decomposition of a closed symmetric 2-differential into a product of closed meromorphic 1-forms is induced by a fibration onto a curve of genus at least two, implying in particular that the fundamental group is large (Thm. 3.3 \cite{1}).


Due to the non-trivial topological restrictions implied by the existence of closed symmetric differentials, a natural problem is to characterize them among all symmetric differentials. This problem is most natural for a surface, since in dimension three or higher a symmetric differential does not in general decompose as a product of 1-forms even pointwise. As explained at the end of section 2.2 of \cite{1}, for $X$ a complex surface, dimension estimates on the jet bundle of sections of $S^n(\Omega^1_X)$ imply that locally every non-degenerate closed symmetric differential satisfies a non-linear differential equation of order depending only on the rank $n$. However, these local differential equations are fully understood only for symmetric differentials of rank one or two, in which case they admit an elegant coordinate-free description. A closed symmetric $1$-differential is the same thing as a closed $1$-form, so the characterizing differential operator is the exterior derivative. A symmetric $2$-differential on a complex surface can be viewed as a complexified Riemannian metric, and the condition that it is closed is equivalent to the existence of Euclidean coordinates, in other words the metric must have vanishing curvature:
\begin{thm}[\cite{1} 2.1, \cite{2} 2.1] \label{thm:BO} Let $X$ be a complex projective surface, and $\omega\in H^0(X,S^n(\Omega^1_X))$ be a symmetric $m$-differential having maximal rank at the generic point. Then
\begin{enumerate}

\item If $n=1$ then $\omega$ is closed if and only if $d\omega=0$.
\item If $n=2$ then $\omega$ is closed if and only if $D_2\omega=0$ where $D_2:H^0(X,S^2(\Omega^1_X))\to H^0(X,K_X^4)$ is the non-linear differential operator given by $D_2(\omega)=det(\omega)^2R(\omega)$, where $R$ is the natural complexification of the Riemannian curvature of $\omega$ viewed as a metric.
\end{enumerate}
\end{thm}

In the higher rank case, although the local existence of differential operators detecting closed symmetric differentials is clear, it is not clear at all how to explicitly represent them in local coordinates, how to understand them in a global, coordinate-free fashion, and what geometric meaning they may carry. The natural question is therefore, in analogy with the rank 1 and 2 cases, whether global operators carrying interesting geometric information exist in general, and whether there are enough of them to characterize closed symmetric differentials:
\begin{qu}[\cite{1} p.10] \textit{Let $X$ be a complex projective surface. Describe the set of non-linear differential operators $$D_m:H^0(X,S^m(\Omega^1_X))\to H^0(X,K_X^{N(m)})$$ whose zero locus is the set of closed symmetric differentials.}
\end{qu}
The authors sketched an approach to the problem along the following lines: let $Y\subset J^kS^m(\Omega^1_X)$ be the subvariety of the bundle of $k$-th jets of symmetric m-differentials defining closed symmetric differentials. This is a proper subvariety as long as $k\geq 2m-2$, and it is naturally invariant under the fibrewise action of the $X$-torsor $G_k$ of $k$-th jets of reparametrization of sections of $S^m(\Omega^1_X)$. Assume that there exists a non-zero function $F$ vanishing on $Y$, which is semi-invariant under the action of $G_k$. Then $F$ defines a differential operator of the expected form. The problem with this approach is that semi-invariant functions are typically constructed profusely using geometric invariant theory of reductive groups, while their existence in this setting is not granted: the group of jet automorphisms of a complex ball is not reductive in general, more precisely $G_k$ is an $X$-torsor under a non-reductive group as long as $k\geq 1$.

In this paper we give a complete answer to the above question in the case $m=3$ by using a different method. The first step is a brute-force construction of differential operators characterizing closed 3-differentials for a specific choice of local coordinates, this is done in Theorem \ref{thm:main1}. Unfortunately the dependence on the choice of local coordinates is not negligible, and the geometry behind the equations we get is rather mysterious. The intuition on how to understand these equations in a coordinate-free fashion comes from web theory (see \cite{4}). A standard normalization for a choice of three 1-forms that define a non-degenerate planar 3-web at the generic point is that they add to zero. The advantage of this normalization is that it is symmetric in the indices, and it determines the 1-forms uniquely up to multiplication by an invertible function. The second step of our proof is thus to rewrite our local equations in terms of local 1-forms $\omega_1,\omega_2,\omega_3$, satisfying  $\omega_1+\omega_2+\omega_3=0$ and such that our 3-differential splits locally as $\eta=\omega_1\omega_2\omega_3$. This normalization, in analogy with the case of planar 3-webs, removes the dependence of our equations on the choice of local coordinates, and leads to the existence of canonical differential operators, acting on symmetric 3-differentials, whose simultaneous vanishing characterizes closed non-degenerate symmetric 3-differentials. More precisely, denoting by $H^0(X,S^3(\Omega^1_X))_{nd}$ the Zariski open subset of $H^0(X,S^3(\Omega^1_X))$ consisting of non-degenerate symmetric 3-differentials, and by $K_X(*)$ the sheaf of forms with poles along a divisor, we obtain a non-linear differential operator $D: H^0(X,S^3(\Omega^1_X))_{nd}\to H^0(X,K_X(*))$ whose zero locus is the set closed symmetric 3-differentials. This is the content of Theorem \ref{thm:main2}

\section{Notation and results}

Let $X$ be a smooth complex manifold of dimension $k$. A {\it symmetric $n$-differential $\eta$} on $X$ is a holomorphic section of the $n$-th symmetric power of the cotangent bundle of $X$. A symmetric $n$-differential $\eta\in H^0(X,S^n(\Omega^1_X))$ determines at each point $x\in X$ a homogeneous degree $n$ polynomial on the tangent space $T_xX$. If $X$ is a surface, then this polynomial factors into linear terms, and hence $\eta$ defines a holomorphic distribution of $n$ lines on $X$, or an $n$-web. A distribution of lines is integrable and can be given by a $1$-form, hence on a surface a symmetric $n$-differential is locally the product of holomorphic $1$-forms. This is not true in general if $k\geq 3$, because the cones that $\eta$ defines on the tangent spaces of $X$ will not in general be unions of hyperplanes, and even if they are the corresponding distributions may fail to be integrable.

A natural question to ask is which symmetric $n$-differentials on a surface can be locally described as products of closed holomorphic $1$-forms. We give a corresponding definition:



\begin{df} A symmetric $n$-differential $\eta$ on a complex surface $X$ is {\it closed} at $x\in X$ if in a neighborhood $U$ of $x$ there exist closed holomorphic $1$-forms $\zeta_1,\ldots,\zeta_n$ such that
\begin{equation}
\eta=\zeta_1\cdots\zeta_n.
\end{equation}

\end{df}

\begin{rem} In the terminology of \cite{2}, differentials having this property are said to have a {\it holomorphic closed decomposition} at $x\in X$, and a closed differential is one that has a holomorphic closed decomposition at a general point $x\in X$. We do not make this distinction because we limit ourselves to studying the local existence question.

\end{rem}




A symmetric $1$-differential $\eta$ is the same thing as a holomorphic $1$-form, and it is closed at $x\in X$ if it is closed in the usual sense, namely if $d\eta$ vanishes in a neighborhood of $x$. A symmetric $2$-differential can be viewed as a complexified Riemannian metric on $X$, and it is closed if the corresponding Riemannian curvature tensor vanishes (see \cite{2} 2.1). In this paper, we give a necessary and sufficient condition for a generic symmetric $3$-differential to be closed. We first restrict ourselves to differentials that define distinct tangent lines at each point:

\begin{df} A symmetric $n$-differential $\eta\in H^0(X,S^n(\Omega_X^1))$ on a complex surface $X$ is called {\it non-degenerate} at $x\in X$ if $\eta$ determines $n$ distinct lines on $T_xX$, and it is called {\it non-degenerate} if it is non-degenerate at every point of $X$.
\end{df}

A non-degenerate symmetric $3$-differential $\eta\in H^0(X,S^3(\Omega^1_X))$ admits near any point $x\in X$ a decomposition
\begin{equation}
\eta=\zeta_1\zeta_2\zeta_3,\quad \zeta_i\wedge\zeta_j\neq 0\mbox{ for }i\neq j,
\end{equation}
where the $\zeta_i$ are holomorphic $1$-forms. This decomposition is not unique, if $\eta=\zeta'_1\zeta'_2\zeta'_3$ is another such decomposition, then after a reordering $\zeta'_i=f_i\zeta_i$ for some non-zero holomorphic functions $f_i$ that satisfy $f_1f_2f_3=1$. We will state our main result in two different forms, depending on how we choose this decomposition.

One possibility is to choose appropriate integrating factors $f_1$ and $f_2$ to make the $1$-forms $\zeta'_1$ and $\zeta'_2$ closed. We can then integrate them to obtain a coordinate system near a point $x\in X$. This proves the following proposition:

\begin{prop} Let $\eta$ be a non-degenerate symmetric $3$-differential on $X$. Then near any point $x\in X$ there exist local coordinates $(z,w)$, and non-zero holomorphic functions $a(z,w)$ and $b(z,w)$, such that near $x$
\begin{equation}
\eta=[a(z,w)dz+b(z,w)dw]dz\,dw.
\label{eq:abcoords}
\end{equation}
\end{prop}

We can now state our main result in terms of this representation.

\begin{thm} \label{thm:main1} Let $\eta$ be a non-degenerate symmetric $3$-differential on a complex surface $X$, which we represent in the form (\ref{eq:abcoords}) near a point $x\in X$. Let
\begin{equation}
A(z,w)=\pa_z\pa_w\log a,\quad B(z,w)=\pa_z\pa_w \log b,
\end{equation}
and suppose that $A-B$ does not vanish near $x$. Define the function
\begin{equation}
\xi(z,w)=\frac{1}{A-B}\left[a\left(\frac{b}{a}B\right)_z-b\left(\frac{a}{b}A\right)_w\right].
\label{eq:ABxi}
\end{equation}
Then $\eta$ is closed at $x\in X$ if and only if the following conditions are satisfied:
\begin{equation}
\left(\frac{\xi}{a}\right)_z=A,\quad \left(\frac{\xi}{b}\right)_w=B.
\label{eq:thm1}
\end{equation}
\end{thm}

\begin{proof} Suppose that $\eta=\zeta_1\zeta_2\zeta_3$ near $x\in X$ where the $\zeta_i$ are closed. These $1$-forms are closed and proportional to $dz$, $dw$ and $adz+bdw$,  so after reordering we can assume that
\begin{equation}
\zeta_1=Z(z)dz,\quad \zeta_2=W(w)dw,\quad \zeta_3=dH(z,w).
\end{equation}
for some nonzero holomorphic functions $Z(z)$, $W(w)$ and $H(z,w)$. Hence we see that  $\eta$ is closed if and only if $a$ and $b$ can be represented in the following form:
\begin{equation}
a(z,w)=Z(z)W(w)H_z(z,w),\quad b(z,w)=Z(z)W(w)H_w(z,w).
\end{equation}
Plugging this into the expression for $\xi$, we find that
\begin{equation}
\xi(z,w)=Z(z)W(w)H_{zw}(z,w).
\end{equation}
It follows that
\begin{equation}
\left(\frac{\xi}{a}\right)_z=\left(\frac{H_{zw}}{H_z}\right)_z=\pa_z\pa_w\log H_z=A,
\end{equation}
\begin{equation}
\left(\frac{\xi}{b}\right)_w=\left(\frac{H_{zw}}{H_w}\right)_w=\pa_w\pa_z\log H_w=B.
\end{equation}
Conversely, suppose that  $a$ and $b$ are nonzero functions such that $A$, $B$ and $\xi$ defined by (\ref{eq:ABxi}) satisfy equations (\ref{eq:thm1}). Comparing the equations, we see that
\begin{equation}
A=\left(\frac{a_w}{a}\right)_z=\left(\frac{\xi}{a}\right)_z,\quad B=\left(\frac{b_z}{b}\right)_w=
\left(\frac{\xi}{b}\right)_w.
\end{equation}
Hence, we can integrate and find functions $c(w)$ and $d(z)$ such that
\begin{equation}
a_w(z,w)=\xi(z,w)+a(z,w)c(w),\quad b_z(z,w)=\xi(z,w)+b(z,w)d(z).
\end{equation}
Define $Z(z)$ and $W(w)$ such that $Z'/Z=d$ and $W'/W=c$, and let
\begin{equation}
a(z,w)=Z(z)W(w)F(z,w),\quad b(z,w)=Z(z)W(w)G(z,w).
\end{equation}
Plugging this into the above equations, we get
\begin{equation}
ZWF_w=\xi=ZWG_z,
\end{equation}
hence we can integrate again and find $H(z,w)$ such that $F(z,w)=H_z(z,w)$ and $G(z,w)=H_w(z,w)$, as required.
\end{proof}

Theorem \ref{thm:main1} gives a complete answer to the question that we pose, but this answer is unsatisfactory on two counts. First of all, the nonlinear PDE (\ref{eq:thm1}) describing closed symmetric 3-differentials is given with respect to a specific choice of coordinate chart (\ref{eq:abcoords}). It is possible to represent this PDE in terms of the coefficients of $\eta$ with respect to an arbitrary choice of local coordinates, but that would involve using the Cardano formula to pass to local coordinates of the form (\ref{eq:abcoords}), and the resulting equations will be very unwieldy.  More importantly, there is no evident geometric meaning to any of the quantities defined in Theorem \ref{thm:main1}. In order to give a coordinate-free version of Theorem \ref{thm:main1}, we instead use some constructions coming from web theory (see \cite{4}).

Let $\eta\in H^0(X,S^3(\Omega^1_X))$ be a non-degenerate symmetric $3$-differential on a complex surface $X$. For any point $x\in X$, there is an open neighborhood $U$ of $x$ and holomorphic $1$-forms $\omega_1,\omega_2,\omega_3$ such that $\eta$ admits the following decomposition on $U$:
\begin{equation}
\eta=\om_1\om_2\om_3,\quad \om_1+\om_2+\om_3=0.
\label{eq:omega}
\end{equation}
The $1$-forms $\om_i$ are well-defined up to permutation and global multiplication by a cube root of unity. The non-degeneracy condition on $\eta$ implies that the $2$-form
\begin{equation}
\Om=\om_1\wedge\om_2=\om_2\wedge\om_3=\om_3\wedge\om_1,
\label{eq:Omega}
\end{equation}
which is well-defined by $\eta$ up to a sixth root of unity, does not vanish on $U$. Therefore we can define holomorphic functions $h_i$ on $U$ as follows:
\begin{equation}
d\om_i=h_i\Om, \quad h_1+h_2+h_3=0.
\end{equation}
We now define the $1$-form $\ga$ as
\begin{equation}
\ga=h_2\om_1-h_1\om_2=h_3\om_2-h_2\om_3=h_1\om_3-h_3\om_1,\quad d\om_i=\ga\wedge\om_i.
\label{eq:ga}
\end{equation}
It is easy to check that $\ga$ is independent of the choice of the cube root of unity and of the ordering of the $\omega_i$.

\begin{rem} A holomorphic $1$-form on a complex surface $X$ determines a dimension one foliation, so a symmetric $3$-differential determines a triple of foliations, or a {\it $3$-web} on $X$. Conversely, a $3$-web determines three one-forms $\om_1,\om_2,\om_3$, which are unique up to multiplication by an invertible function if we impose the condition $\om_1+\om_2+\om_3=0$. The form $\ga$ is then independent of the choice of the function, and its exterior derivative $d\ga$ is called the {\it Blaschke curvature} of the web. Geometrically, the Blaschke curvature measures the non-closedness of a small hexagonal trajectory made by following the leaves of the foliation around a given point.

We also remark that if $\eta$ is closed, then its decomposition as a product of closed $1$-forms is not necessarily the same as the decomposition (\ref{eq:omega}).

\end{rem}


\begin{thm} \label{thm:main2} Let $\eta$ be a non-degenerate symmetric $3$-differential on a complex surface $X$, and let $\omega_i$, $\Omega$, $h_i$ and $\gamma$ be defined as above. Suppose that the Blaschke curvature $d\gamma$ of $\eta$ does not vanish at $x\in X$. Define the following $1$-forms:
\begin{equation}
\beta_{ijkl}=\frac{1}{d\gamma}\left[2\frac{d(h_i\omega_j)}{\Omega}d\omega_k+d\left(\frac{d(h_i\omega_j)}{\Omega}\right)\wedge \omega_k\right] \omega_l,
\end{equation}
\begin{equation}
\alpha=\displaystyle\sum_{i\neq j}h_i\omega_j,\quad \beta_s=-\beta_{2121}-\beta_{1212}+\beta_{2211}+\beta_{1122},
\end{equation}
\begin{equation}
\beta_a=2\beta_{1121}-2\beta_{1211}+\beta_{2121}-\beta_{1212}+\beta_{1122}-\beta_{2211}+2\beta_{2122}-2\beta_{2212}.
\end{equation}
Then $\eta$ is closed if and only if the $1$-forms $\beta_a-\alpha$ and $\beta_s-\gamma$ are closed:
\begin{equation}
d\beta_a-d\alpha=0,\quad d\beta_s-d\gamma=0.
\label{eq:thm3}
\end{equation}
\end{thm}

\begin{rem} The $1$-form $\beta_s$, like $\gamma$, does not depend on any of the choices made when defining the $\omega_i$, while the $1$-forms $\beta_a$ and $\alpha$ change sign if we transpose any two of the $\omega_i$. In other words, the first and second equations in (\ref{eq:thm3}) are well-defined equations between sections of $(\Omega_X^2)^{\otimes 2}$ and $\Omega_X^2$, respectively.

\end{rem}

\begin{proof} The proof is essentially a long calculation that verifies that condition (\ref{eq:thm3}) is equivalent to condition (\ref{eq:thm1}). Let $dz$ and $dw$ be closed $1$-forms that are multiples of $\omega_1$ and $\omega_2$, respectively, then $(z,w)$ is a coordinate system near $x\in X$ in terms of which the $\omega_i$ have the form
\begin{equation}
\omega_1=f(z,w)dz,\quad \omega_2=g(z,w)dw,\quad \omega_3=-f(z,w)dz-g(z,w)dw,
\end{equation}
where $f$ and $g$ are nonzero holomorphic functions. Let 
\begin{equation}
\quad F=\pa_z\pa_w \log f,\quad G=\pa_z\pa_w\log g,
\end{equation}
then $\eta$ and the objects defined in (\ref{eq:Omega})-(\ref{eq:ga}) can be expressed in terms of $f$ and $g$ in the following way:
\begin{equation}
\eta=-f^2gdz^2dw-fg^2dzdw^2,\quad h_1=-\frac{f_w}{fg},\quad h_2=\frac{g_z}{fg},\quad h_3=\frac{f_w-g_z}{fg},
\end{equation}
\begin{equation}
\quad\Omega=fg\,dz\wedge dw,\quad \gamma=\frac{g_z}{g}dw+\frac{f_w}{f}dw,\quad
d\gamma=(F-G)dz\wedge dw.
\end{equation}
Comparing with $a$ and $b$, we see that
\begin{equation}
a=-f^2g,\quad b=-fg^2,\quad A=2F+G,\quad B=F+2G.
\end{equation}
\begin{rem} The observation that the quantity $A-B=F-G$ occurring in the denominator of $\xi$ in (\ref{eq:ABxi}) is proportional to the Blaschke curvature of the web defined by $\eta$ was what originally led us to rephrase the problem in terms of web theory.
\end{rem}
In order to phrase equations (\ref{eq:thm1}) in terms of the web data, we need to find an expression for $\xi$. After some guesswork, we see that
\begin{equation}
\xi=f^2g^2\Xi,
\end{equation}
where the function $\Xi$ is defined as
\begin{equation}
\Xi=
\frac{1}{d\gamma}\left[2\frac{d(h_1\omega_2+2h_2\omega_1)}{\Omega}d \omega_2+
2\frac{d(2h_1\omega_2+h_2\omega_1)}{\Omega}d\omega_1+\right.
\end{equation}
$$
\left.+d\left(\frac{d(h_1\omega_2+2h_2\omega_1)}{\Omega}\right)\wedge \omega_2+d\left(\frac{d(2h_1\omega_2+h_2\omega_1)}{\Omega}\right)\wedge \omega_1\right].
$$
Moreover, equations (\ref{eq:thm1}) can be rewritten as
\begin{equation}
d(\Xi \omega_2)=d(2h_1\omega_2+h_2\omega_1),\quad -d(\Xi \omega_1)=d(h_1\omega_2+2h_2\omega_1).
\end{equation}
We now observe that if we had instead picked $\omega_2$ and $\omega_3$ to be proportional to $dx$ and $dy$, we would have obtained an equivalent system of equations, because the condition that $\eta$ is closed does not depend on a choice of coordinate system. Therefore, we can symmetrize the above equation with respect to a cyclic permutation of the $\omega_i$, and a calculation shows that the resulting equations are precisely (\ref{eq:thm3}), which completes the proof.

\end{proof}

\begin{rem} Theorem \ref{thm:main2} describes a necessary and sufficient condition for a symmetric $3$-differentials $\eta$ to be closed as the vanishing of a certain tensor determined by $\eta$. This result lines up with Theorem \ref{thm:BO}, which describes closed symmetric $1$- and $2$-differentials by the vanishing of their exterior derivative and curvature tensor, respectively. However, the latter two objects can be thought of more specifically as the curvatures of a connection on an appropriate vector bundle. We do not know of a similar explicit geometric interpretation of the tensor defined in Theorem \ref{thm:main2}.
\end{rem}

\noindent \textbf{Acknowledgments.} The authors are thankful to Fedor Bogomolov for suggesting the problem, and for a number of interesting conversations on his theory of closed symmetric differentials.

\end{document}